\documentclass{amsart}                          
\usepackage{epsfig}            
\usepackage{amssymb,amscd,bbm}     
\usepackage{amsthm,amsgen,amsmath,empheq}
\usepackage{xypic}       
\usepackage{mathdots}
\usepackage{stmaryrd}

\xyoption{all}                
\xyoption{arc}

\newcommand{\longgraph}[4]{ \ensuremath{
 \begin{aligned}
 \begin{xy}                      
  (0,-2)*[o]+=<7pt>{\scriptstyle #1}="a",
  (6,2)*[o]+=<7pt>{\scriptstyle #2}="b", 
  (12,-2)*[o]+=<7pt>{\scriptstyle #3}="c",
  (18,2)*[o]+=<7pt>{\scriptstyle #4}="d",
  "a";"b"**\dir{-}?>*\dir{>},    
  "b";"c"**\dir{-}?>*\dir{>},    
  "c";"d"**\dir{-}?>*\dir{>}     
 \end{xy}                        
 \end{aligned}
} }

\newcommand{\graphpp}[3]{ \ensuremath{
 \begin{aligned}
 \begin{xy}                      
  (0,-2)*[o]+=<7pt>{\scriptstyle #1}="a",
  (6,2)*+=<7pt>{\scriptstyle #2}="b", 
  (12,-2)*+=<7pt>{\scriptstyle #3}="c",
  "a";"b"**\dir{-}?>*\dir{>},    
  "b";"c"**\dir{-}?>*\dir{>}     
 \end{xy}                        
 \end{aligned}
} }
\newcommand{\linep}[2]{ \ensuremath{   
 \begin{aligned}
 \begin{xy}     
	(0,-2)*[o]+=<8pt>{\scriptstyle #1}="a",
	(6,2)*[o]+=<8pt>{\scriptstyle #2}="b", 
  "a";"b"**\dir{-}?>*\dir{>},   
 \end{xy}   
 \end{aligned}
} }
\newcommand{\dlinep}[2]{ \ensuremath{   
 \begin{aligned}
 \begin{xy}     
	(0,2)*[o]+=<8pt>{\scriptstyle #1}="b",
	(6,-2)*[o]+=<8pt>{\scriptstyle #2}="c", 
  "b";"c"**\dir{-}?>*\dir{>},   
 \end{xy}   
 \end{aligned}
} }

\newcommand{\lgbrack}{\lfloor\hspace{-.27em}\lfloor} 
\newcommand{\rgbrack}{\rfloor\hspace{-.27em}\rfloor}

\theoremstyle{plain}                          
\newtheorem{theorem}{Theorem}[section]                          
\newtheorem{proposition}[theorem]{Proposition}                          
\newtheorem{lemma}[theorem]{Lemma}                          
\newtheorem{corollary}[theorem]{Corollary}

\theoremstyle{definition}                          
\newtheorem{definition}[theorem]{Definition}  

\theoremstyle{remark}                          
\newtheorem{remark}[theorem]{Remark}

\newtheorem{example}[theorem]{Example}

\begin{document}

\title{The left-greedy Lie algebra basis and star graphs}

\author[B. Walter]{Benjamin Walter} 
\address{
Department of Mathematics \\ Middle East Technical University, Northern Cyprus Campus \\
Kalkanli, Guzelyurt, KKTC, Mersin 10 Turkey
}
\email{benjamin@metu.edu.tr}

\author[A. Shiri]{Aminreza Shiri}
\thanks{
This work is the result of an undergraduate research project of the second author 
at the Middle East Technical University, Northern Cyprus Campus during the summer 
and fall of 2014.
}



\begin{abstract}
 We construct a basis for free Lie algebras via a ``left-greedy'' bracketing algorithm on
 Lyndon-Shirshov words.  We use a new tool -- the configuration
 pairing between Lie brackets and graphs of Sinha-Walter -- to show that the
 left-greedy brackets form a basis.  Our constructions further equip 
 the left-greedy brackets 
 with a dual monomial Lie coalgebra basis of ``star'' graphs. 
 We end with a brief example using the dual basis of star graphs in a 
 Lie algebra computation.
\end{abstract}

\maketitle

\section{Introduction}

Lie algebras are classical objects with applications in differential geometry, theoretical physics, and computer science.
A Lie algebra is a vector space which has an extra non-associative (bilinear) operation called a
``Lie bracket'', written $[a,b]$.  The Lie bracket operation satisfies anti-commutativity 
and Jacobi relations.
\[
\begin{aligned}
\text{(Anti-commutativity)} \quad & \quad
    0 = [a,b] + [b,a]  \\
\text{(Jacobi)} \quad & \quad 
    0 = [a,[b,c]] + [c,[a,b]] + [b,[c,a]] 
\end{aligned}
\]
A free Lie algebra is a Lie algebra whose bracket operation satisfies no extra
relations -- only the two written above and any relations which can be generated by
combining them together. For example
\[
   [a,[b,c]] - [[a,b],c] = [[c,a],b]
\]
is a relation for free Lie algebras since \( [c,[a,b]] = -[[a,b],c]\) by 
anti-commutativity, and similarly for \([b,[c,a]]\).
Free Lie algebras are fundamental in that every Lie algebra can be written 
``via generators and relations'' as a free Lie algebra
with further extra relations placed on its bracket operation.

Recall that a set of elements generates an algebra if all other elements in the algebra can be
obtained via sums of products of elements from the set.  A minimal generating set is 
called an 
algebra basis.  We are interested in linear bases for  
algebras -- these are 
minimal sets consisting of an algebra basis along with enough products of these so 
that all further algebra elements
can be reached using only sums.

The current work describes a new linear basis for
free Lie algebras; along with a new method for finding, 
computing with, and proving theorems about general free Lie algebra
bases.  Our method uses the graph/tree pairing developed in
\cite{Sin05} and \cite{Sin06}, which yields a new way to 
describe Lie coalgebras via graphs as applied in
\cite{SiWa11} and explained further in \cite{Wal10}.
Since we introduce a new and different way to perform
calculations in free Lie algebras, we give many detailed examples throughout.
For readers interested in a history of bases of free Lie algebras, we suggest
\cite[\S 3.2]{BoCh06}.

\medskip

We would like to thank the reviewer for a careful and detailed reading of our submission and 
also for suggesting that we consider an application of Lazard elimination (see Remark~\ref{R:Lazard})
and rewriting systems on Lyndon words (see Remark~\ref{R:rewrite}).

\section{Notation and Classical Constructions}

In this paper we will say ``alphabet'' for a collection of abstract letters (or variables).
A ``word'' in an alphabet is a (non-commutative, associative) string (or product) of letters 
from the alphabet.
An ordering on an alphabet 
(such as the standard alphabetical ordering in English) 
induces an ordering on words called the ``lexicographical'' (or dictionary) ordering.
The cyclic permutations of a word are given by removing letters from the beginning of the word 
and moving them to the end.

\begin{example}
The cyclic permutations of the word $abcd$ are
$bcda$, $cdab$, $dabc$. \\
The cyclic permutations of the word $aaabb$ are
$aabba$, $abbaa$, $bbaaa$ and $baaab$. 
\end{example}

Classically, a Lyndon-Shirshov word \cite{ChFoLy58} \cite{Shir09} (often called a ``Lyndon word'') is a word which is
lexicographically less than all of its cyclic permutations.  There are several methods to
form bases for free Lie algebras using Lyndon-Shirshov words
\cite{MeRe89} \cite{Reu93} \cite{Chi06} \cite{Sto08}.
For example, the standard bracketing \cite{Reu93} of an
Lyndon-Shirshov word $w$ is written $[w]$, given by splitting the word $w$ into two (nonempty) sub-words 
$w = uv$,
such that the subword $v$ is a maximally long Lyndon-Shirshov word, and
then recursively defining $[w] = [ [u], [v] ]$ (where the bracketing of a single letter word
is itself, $[a] = a$).  The collection of all standard bracketings of Lyndon-Shirshov words 
gives a linear
basis for the free Lie algebra on the underlying alphabet.

\begin{example}
The word $aaabb$ is a Lyndon-Shirshov word, but not $aabba$ or $abbaa$.  The word $abab$ is not 
a Lyndon-Shirshov word -- since it is a cyclic permutation of itself, it is not less than all of its
cyclic permutations.  The standard bracketing of $aaabb$ is
\begin{align*} 
[aaabb] &= \bigl[\, [a],\ [aabb]\, \bigr] \\  
	&= \bigl[a,\,\bigl[\ [a],\ [abb]\,\bigr]\bigr]  \\
    &= \bigl[a,\,\bigl[a,\,\bigl[\,[ab],\ [b]\,\bigr]\bigr]\bigr]  \\
    &= \bigl[a,\,\bigl[a,\,\bigl[\bigl[a,\,b\bigr],\,b\bigr]\bigr]\bigr]
\end{align*}
\end{example}

The standard bracketing of a Lyndon-Shirshov word can also be described recursively from the
inner-most brackets to the outer-most roughly as follows:  Read the letters of a Lyndon-Shirshov word from 
right to left looking for the first occurrence of consecutive letters $...a_ia_{i+1}...$ where
$a_i < a_{i+1}$ (called the ``right-most inversion'').  Replace $a_ia_{i+1}$ by the bracket $[a_i,a_{i+1}]$ 
which we will consider to be a new ``letter'' placed in the ordered alphabet in the lexicographical
position of the word ``$a_ia_{i+1}$''.
Repeat. (For a more detailed description see \cite[\S 2]{MeRe89}.)

Our construction of left-greedy brackets will also proceed from the inner-most bracket to outer-most
bracket, similar to the ``rewriting system'' presented above.  
However, just as with the standard bracketing,
left-greedy brackets can also be described from the outer-most to inner-most brackets.

\section{Left-Greedy Brackets and Star Graphs}

\subsection{Simple words}

\begin{definition}
Given a fixed letter $a$ in an alphabet, 
an $a$-simple word is a word of the form $w = aa\cdots ax$ (written $w = a^n x$ for short) where
$x$ is any single letter not equal to $a$.
The single-letter word, $w = x$ (i.e. $w=a^0x$), is also an $a$-simple word (for $x\neq a$).
\end{definition}

The collection of all words in an alphabet is itself an (infinite) ordered alphabet 
(with the lexicographical ordering).  A word in the alphabet whose letters are words in 
another alphabet will be casually referred to as a ``word of words.''  Note that such an 
expression of a word as a product of subwords is equivalent to partition of the word.  
Considering words as ordered sets of letters, partitions are order-preserving surjections 
of sets; hence our notation for partitioning a word will be a double-headed arrow
$\twoheadrightarrow$. 


\begin{remark}\label{R:rewrite}
A partition of a word is equivalent to a rewriting (cf \cite{Mel97})
which combines multiple subwords in parallel.  For our construction and proofs
we will critically make use of the levels of nesting of partitions.  We use
the term ``partition'' so that both our 
notation and our terminology reflect this emphasis.
\end{remark}

\begin{definition}
A simple partition of a word $w$ is an expression of $w$ as subwords
$w = \alpha_1 \alpha_2 \ldots \alpha_k$ where each $\alpha_i$ is an $a$-simple word 
and $a$ is the first letter of $w$.  We will write 
$w\twoheadrightarrow \alpha_1 \alpha_2 \ldots \alpha_k$.
\end{definition}

Note that words have at most one simple partition.
The subword $\alpha_1$ must consist of the initial string of $a$'s as well as the first non-$a$ letter of $w$.
If the letter in $w$ following $\alpha_1$ is $a$, then $\alpha_2$ must consist of the next string of $a$'s as well
as the next non-$a$ letter.  If the letter following $\alpha_1$ is not $a$, then $\alpha_2$ will 
consist of only that one letter.  (See the first line of the examples in~\ref{E:simple words}.)

\begin{remark}\label{R:Lazard}
A simple partition of a word is equivalent to performing Lazard elimination \cite[Ch.5]{Loth97} on
the word, eliminating the first letter $a$ via the bissection 
$\bigl(a^*(A\backslash a),\,a\bigr)$.  
It appears likely that the constructions of nested partitions and fully partitioned words
which will follow may also be performed via a recursive series of Lazard elimination steps along
the lines of: ``Order all words lexicographically. Eliminate them, one at at time, beginning
with the least ordered word $a$.'' 
[Possibly it will be best to restrict to words of length $\le n$ at first.]

This would give an alternate proof that the left-greedy 
brackets form a basis.  However, making the previous statement precise and showing
that it gives a well-defined recursive operation which will terminate is complicated. 
Also, following such a path would not yield the dual basis of star graphs, which we wish to
exploit in later work.
\end{remark}

Given a simple partition $w \twoheadrightarrow \alpha_1\alpha_2 \ldots \alpha_k$, 
we may recurse: The $a$-simple subwords $\alpha_1$, $\alpha_2$, etc. are letters in the alphabet
of words. They may have a further simple partition (now as $\alpha_1$-simple words). 
This process constructs a unique nested partition of a word such that each nested level  
is a simple partition.

\begin{definition}
A word fully partitions if it has a series of simple partitions
\[ w \twoheadrightarrow \omega_1 \twoheadrightarrow \cdots \twoheadrightarrow \omega_\ell \]
where $\omega_\ell$ is the trivial coarse partition. 
\end{definition}

Colloquially, a word fully partitions if it is a simple word of simple words of simple words of etc.

\begin{example}\label{E:simple words}
Words fully partition as follows (for clarity we will use distinct delimiters (, [, and \{ to indicate
different nested levels of partition.)

$\bullet \quad aaaab 
	\twoheadrightarrow (aaaab)$

\vspace{.1em}
    
$\bullet \quad 
  \begin{aligned}[t]
	ababb 
	&\twoheadrightarrow (ab)\;(ab)\;(b)  \\
	&\twoheadrightarrow \bigl[ (ab)(ab)(b) \bigr] 
  \end{aligned}$ 
  
\vspace{.1em}
    
$\bullet \quad 
  \begin{aligned}[t]
	aabcb 
	&\twoheadrightarrow (aab)\;(c)\;(b)  \\
    &\twoheadrightarrow \bigl[ (aab)(c) \bigr] \; \bigl[(b)\bigr] \\
    &\twoheadrightarrow \Bigl\{ \bigl[ (aab)(c) \bigr] \bigl[(b)\bigr] \Bigr\}
  \end{aligned}$ 
  
\vspace{.1em}
    
$\bullet \quad 
  \begin{aligned}[t]
	ababbabaab  
	&\twoheadrightarrow (ab)\;(ab)\;(b)\;(ab)\;(aab)  \\
    &\twoheadrightarrow \bigl[ (ab)(ab)(b) \bigr]\;\bigl[(ab)(aab)\bigr] \\
    &\twoheadrightarrow \Bigl\{ \bigl[ (aab)(b) \bigr] \bigl[(ab)(aab)\bigr] \Bigr\}
  \end{aligned}$
    
\noindent    
For visual clarity, we have found that indicating nested partitions via underlining is 
often more understandable than using nested parentheses.

$\bullet \quad aaaab 
	\twoheadrightarrow \underline{aaaab}$ 
    
$\bullet \quad ababb 
	\twoheadrightarrow \underline{\underline{ab}\;\underline{ab}\;\underline{b}} 
    \vphantom{\bigl)}$ 
    
$\bullet \quad aabcb 
	\twoheadrightarrow \underline{\underline{\underline{aab}\;\underline{c}}\;\underline{\underline{b}}} 
    \vphantom{\Bigl)}$ 
    
$\bullet \quad ababbabaab 
	\twoheadrightarrow \underline{\underline{\underline{ab}\;\underline{ab}\;\underline{b}}\;
                       \underline{\underline{ab}\;\underline{aab}}} 
    \vphantom{\Bigl)}$ 
    
$\bullet \quad abcabcabbabcaab 
	\twoheadrightarrow \underline{
    		\underline{\underline{\underline{ab}\;\underline{c}}\;
                       \underline{\underline{ab}\;\underline{c}}\;
                       \underline{\underline{ab}\;\underline{b}}}\;
            \underline{\underline{\underline{ab}\;\underline{c}}\;
                       \underline{\underline{aab}}}}
    \vphantom{\Bigl)}$ 
\end{example}

\begin{example}
Some words do not fully partition. 

$\bullet \quad aaaa$ \quad contains repetitions of only one letter. 

$\bullet \quad aaba$ \quad has the same initial and final letter. 

$\bullet \quad abab \twoheadrightarrow (ab)(ab)$ \quad which is repetitions of a single subword $(ab)$. 

$\bullet \quad abaabab \twoheadrightarrow (ab)(aab)(ab)$ \quad  which has the same initial and final subword $(ab)$.

$\bullet \quad ababbcababb \twoheadrightarrow\bigl[(ab)(ab)(b)\bigr]\bigl[(c)\bigr]\bigl[(ab)(ab)(b)\bigr]$ \quad
	which has the same initial and final subword
    $\bigl[(ab)(ab)(b)\bigr]$. 
\end{example}

The following simple lemma follows immediately from standard facts about Lyndon-Shirshov words.  We
give a proof below for completeness.

\begin{lemma}
Every Lyndon-Shirshov word fully partitions.
\end{lemma}

\begin{proof}
Fix a word $w$ with initial letter $a$.
The only obstacle to the partition of $w$ into $a$-simple words is whether the final
letter and the initial letter match.  More generally, each step of the recursive partitioning can be
completed as long as the initial and final subword do not match.  This fails only if the word
has the form $w = \alpha \chi \alpha$ where $\alpha$ and $\chi$ are subwords 
(the subword $\chi$ may be empty and is likely not simple).  

However no Lyndon-Shirshov word has this form.  If $\chi$ is empty then $w=\alpha\alpha$ which is not Lyndon-Shirshov.  
If $\chi$ is
nonempty, then one of the cyclic reorderings of $w$ is lexicographically lower:  
Either $\alpha\alpha\chi < \alpha\chi\alpha$ (if $\alpha<\chi$) or else $\chi\alpha\alpha < \alpha\chi\alpha$ (if $\chi<\alpha$).
\end{proof}

\begin{remark}
Many non-Lyndon-Shirshov words also fully partition.
The requirement that $w\neq \alpha\chi\alpha$ for any subwords $\alpha$ and $\chi$ is much weaker than the Lyndon-Shirshov requirement.
Via some experimentation, we have found that it is possible to use methods similar to those 
presented in the current work to find new bases
for Lie algebras which are constructed from sets of words other than the Lyndon-Shirshov words.  
It is unclear if these sets of words are also bases for the shuffle algebra. 
\end{remark}

\subsection{Left-greedy brackets}

\begin{definition}
 The left-greedy bracketing of the $a$-simple word $w=a^n x$, denoted $\lgbrack w\rgbrack$, is the
 standard right-normed bracketing $\lgbrack a^n x\rgbrack = [a,[a,\ldots[a, [a,x]]\cdots]]$. \\
 The left-greedy bracketing of a simple word of simple words (and, more generally, any fully 
 partitioned word) is defined recursively
 \[
  \lgbrack \alpha^n \chi\rgbrack = 
    \Bigl[\, \lgbrack \alpha \rgbrack\,,\ \Bigl[\, \lgbrack \alpha \rgbrack\,,\ \ldots
    \Bigl[\, \lgbrack \alpha \rgbrack\,,\ \Bigl[\, \lgbrack \alpha \rgbrack,\ \lgbrack \chi \rgbrack 
    \Bigr] \Bigr] \cdots \Bigr] \Bigr]
 \]
\end{definition}

\begin{example}
Following are some examples of left-greedy bracketings of fully partitioned words.  To aid understanding 
in the examples below, we underline to indicate their full partition into simple words. (Note that we
do not require for words to be Lyndon-Shirshov in order to define their left-greedy bracketing.)

$\bullet \quad \lgbrack \,
    \underline{aaab} 
    \, \rgbrack = 
	\Bigl[ a,\;\bigl[ a,\;[a,\;b]\bigr]\Bigr]$ 
    
$\bullet \quad \lgbrack\,  
    \underline{\underline{ab}\,\underline{ab}\,\underline{b}} 
    \, \rgbrack = 
	\Bigl[ [a,b],\ \bigl[ [a,b],\;b \bigr]\Bigr]$
    
$\bullet \quad \lgbrack\,  
	\underline{\underline{\underline{aab}\,\underline{c}}\;\underline{\underline{b}}} 
    \, \rgbrack =
	\Bigl[\bigl[ [a,[a,b]],\ c\bigr],\ b\Bigr]$
    
$\bullet \quad \lgbrack\,  
	\underline{
	 \underline{\underline{ab}\,\underline{ab}\,\underline{b}}\;
     \underline{\underline{ab}\,\underline{aab}} 
    }
    \, \rgbrack =
	\Bigl[ \bigl[[a,b],\ \bigl[[a,b],\;b\bigr]\bigr],\ \bigl[[a,b],\;[a,[a,b]] \bigr]\Bigr]$
\end{example}

\begin{remark}
The name ``left-greedy'' is due to the fact that the bracketing of the word $aaabcd$ begins with
inner-most bracket $[a,b]$ and then brackets leftwards -- $[a,[a,[a,b]]]$ before bracketing to the 
right.  An alternative ``right-greedy'' bracketing, would go to the right $[[[a,b],c],d]$ before 
bracketing leftwards.  Both of these yield free Lie algebra bases, but the left-greedy bracketing 
has a cleaner basis proof and appears to have better properties.  We leave the discussion of 
the beneficial properties
of the left-greedy bracketing to a later paper.
\end{remark}

\begin{remark}
Left-greedy bracketing of Lyndon-Shirshov words is different than other bracketing methods considered in the 
literature.  We will give a few examples here for quick comparison with some other similar 
methods.  Consider the Lyndon-Shirshov word $w = aababb$.

$\bullet \quad \lgbrack aababb \rgbrack = [[[a,[a,b]],[a,b]],b] $ \quad 
		the left-greedy bracketing.
        
$\bullet \quad \ [aababb] = [a,[[a,b],[[a,b],b]]] $ \quad 
		the standard Lyndon-Shirshov-bracketing \cite{Reu93}.
        
$\bullet \quad \ \llbracket aababb \rrbracket = [[a,[a,b]],[[a,b],b]] $ \quad 
		the bracketing of Chibrikov \cite[\S 4]{Chi06}.
%
\end{remark}

\subsection{Star graphs}

By a ``graph'' we mean a finite, directed graph whose vertices are labelled by letters.

\begin{definition}
The star graph of the $a$-simple word $w=a^n x$, denoted $\bigstar(w)$, is the graph with $n$ vertices
labeled $a$, one vertex labeled $x$, and an edge from each $a$ vertex to the $x$ vertex.
The $x$ vertex is called the ``anchor vertex''.
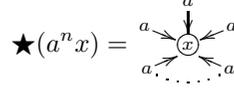
\[
\bigstar (a^nx) = 
\begin{xy}
  (0,0)*[o]+=<9pt>{\scriptstyle x}*\frm{o}="c", 
  (-5.5,-3.2)*[o]+=<7pt>{\scriptstyle a}="a1", 
  (5.5,-3.2)*[o]+=<7pt>{\scriptstyle a}="a2", 
  (-5.8,2.4)*[o]+=<7pt>{\scriptstyle a}="a3", 
  (5.8,2.4)*[o]+=<7pt>{\scriptstyle a}="a4", 
  (0,5.7)*[o]+=<7pt>{\scriptstyle a}="a5",
  "a1";"c"+(-1.5,-1.5)**\dir{-}?>*\dir{>},    
  "a2";"c"+(1.5,-1.5)**\dir{-}?>*\dir{>},    
  "a3";"c"+(-1.7,.7)**\dir{-}?>*\dir{>},
  "a4";"c"+(1.7,.7)**\dir{-}?>*\dir{>},    
  "a5";"c"**\dir{-}?>*\dir{>},
  {\ar@/_5pt/@{.} "a1";"a2"}
\end{xy}
\]
The star graph of a simple word of simple words (and, more generally, any fully partitioned word) 
is defined recursively. 
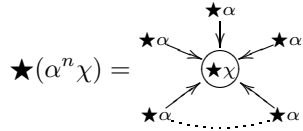
\[
\bigstar (\alpha^n\chi) = 
\begin{xy}
  (0,0)*[o]+=<15pt>{\scriptstyle \bigstar \chi}*\frm{o}="c", 
  (-8.5,-6)*[o]+=<7pt>{\scriptstyle \bigstar \alpha}="a1", 
  (8.5,-6)*[o]+=<7pt>{\scriptstyle \bigstar \alpha}="a2", 
  (-8.8,4)*[o]+=<7pt>{\scriptstyle \bigstar \alpha}="a3", 
  (8.8,4)*[o]+=<7pt>{\scriptstyle \bigstar \alpha}="a4", 
  (0,8)*+U{\scriptstyle \bigstar \alpha}="a5",
  "a1";"c"+(-2.7,-1.7)**\dir{-}?>*\dir{>},    
  "a2";"c"+(2.7,-1.7)**\dir{-}?>*\dir{>},    
  "a3";"c"+(-2.5,1.5)**\dir{-}?>*\dir{>},
  "a4";"c"+(2.5,1.5)**\dir{-}?>*\dir{>},    
  "a5";"c"+(0,3)**\dir{-}?>*\dir{>},
  {\ar@/_5pt/@{.} "a1";"a2"}
\end{xy}
\]
The graph $\bigstar (\alpha^n\chi)$
consists of $n$ disjoint subgraphs $\bigstar (\alpha)$ 
and one disjoint subgraph $\bigstar (\chi)$ 
with edges connecting from the 
anchor vertices of the $\bigstar(\alpha)$ to the anchor vertex of $\bigstar (\chi)$.  
The anchor
vertex of the subgraph $\bigstar(\chi)$ serves as the anchor vertex of the star graph 
$\bigstar(\alpha^n \chi)$.
\end{definition}

\begin{remark}
The star graph of an $a$-simple word consisting of one (non-$a$) letter $w=x$ is a single anchor
vertex.
\[
\bigstar (x) = 
\begin{xy}
  (0,0)*+(1.5,1.5){\scriptstyle x}*\frm{o}="c", 
\end{xy}
\]
\end{remark}

\begin{example}
Following are some examples of star graphs.  In the examples below, the anchor vertex of the
subgraphs are indicated with dotted circles and the anchor vertex of the entire graph is 
indicated with a solid circle.
    
$\bullet \quad \bigstar(
 \underline{aaab} 
 ) = 
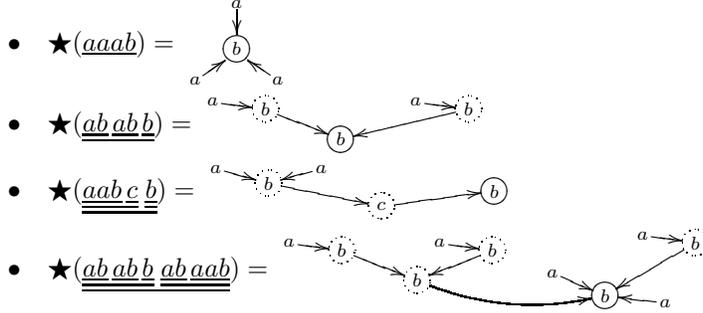
\begin{aligned}
\begin{xy}
  (0,0)*[o]+=<10pt>{\scriptstyle b}*\frm{o}="c", 
  (-5.5,-4.2)*[o]+=<7pt>{\scriptstyle a}="a1", 
  (5.5,-4.2)*[o]+=<7pt>{\scriptstyle a}="a2", 
  (0,6)*+U{\scriptstyle a}="a5",
  "a1";"c"+(-1.5,-1,5)**\dir{-}?>*\dir{>},    
  "a2";"c"+(1.5,-1.5)**\dir{-}?>*\dir{>},    
  "a5";"c"+(0,2)**\dir{-}?>*\dir{>},
\end{xy}
\end{aligned}$
 
$\bullet \quad \bigstar( 
 \underline{\underline{ab}\,\underline{ab}\,\underline{b}} 
 ) = 
\begin{aligned}
\begin{xy}
  (0,0)*[o]+=<10pt>{\scriptstyle b}*\frm{o}="c",
  (17,4)*[o]+=<10pt>{\scriptstyle b}*\frm{.o}="b2",
  (-10,4)*[o]+=<10pt>{\scriptstyle b}*\frm{.o}="b1",
  (-17,5)*[o]+=<7pt>{\scriptstyle a}="a1",
  (10,5)*[o]+=<7pt>{\scriptstyle a}="a2",
  "a1";"b1"**\dir{-}?>*\dir{>},
  "a2";"b2"**\dir{-}?>*\dir{>},
  "b1";"c"**\dir{-}?>*\dir{>},
  "b2";"c"**\dir{-}?>*\dir{>},
\end{xy}
\end{aligned}$

$\bullet \quad \bigstar(
  \underline{\underline{\underline{aab}\,\underline{c}}\;\underline{\underline{b}}} 
  ) = 
\begin{aligned}\begin{xy}
  (15,0)*[o]+=<10pt>{\scriptstyle c}*\frm{.o}="c",
  (30,2)*[o]+=<10pt>{\scriptstyle b}*\frm{o}="b2",
  (0,3)*[o]+=<10pt>{\scriptstyle b}*\frm{.o}="b1",
  (-7,5)*[o]+=<7pt>{\scriptstyle a}="a1",
  (7,5)*[o]+=<7pt>{\scriptstyle a}="a2",
  "a1";"b1"**\dir{-}?>*\dir{>},
  "a2";"b1"**\dir{-}?>*\dir{>},
  "b1";"c"**\dir{-}?>*\dir{>},
  "c";"b2"**\dir{-}?>*\dir{>},
\end{xy}\end{aligned}$

$\bullet \quad \bigstar(
  \underline{
   \underline{\underline{ab}\,\underline{ab}\,\underline{b}}\;
   \underline{\underline{ab}\,\underline{aab}} 
  }
 ) = 
 \begin{aligned}\begin{xy}
  (10,0)*[o]+=<10pt>{\scriptstyle b}*\frm{.o}="c",
  (20,4)*[o]+=<10pt>{\scriptstyle b}*\frm{.o}="b2",
  (0,4)*[o]+=<10pt>{\scriptstyle b}*\frm{.o}="b1",
  (-7,5)*[o]+=<7pt>{\scriptstyle a}="a1",
  (13,5)*[o]+=<7pt>{\scriptstyle a}="a2",
  "a1";"b1"**\dir{-}?>*\dir{>},
  "a2";"b2"**\dir{-}?>*\dir{>},
  "b1";"c"**\dir{-}?>*\dir{>},
  "b2";"c"**\dir{-}?>*\dir{>},
  (35,-2)*[o]+=<10pt>{\scriptstyle b}*\frm{o}="c2",
  (43,-3)*[o]+=<7pt>{\scriptstyle a}="a21",
  (28,1)*[o]+=<7pt>{\scriptstyle a}="a22",
  (47,5)*[o]+=<10pt>{\scriptstyle b}*\frm{.o}="b22",
  (40,6)*[o]+=<7pt>{\scriptstyle a}="a23",
  "a21";"c2"**\dir{-}?>*\dir{>},
  "a22";"c2"**\dir{-}?>*\dir{>},
  "b22";"c2"**\dir{-}?>*\dir{>},
  "a23";"b22"**\dir{-}?>*\dir{>},
  {\ar@/_6pt/ "c";"c2" }
\end{xy}\end{aligned}$
\end{example}

\begin{remark}
The name ``star graph'' comes from imagining the graph $\bigstar(a^n b)$ as a sun
($b$) with planets ($a$) orbiting around it.  The recursive construction of star graphs then 
composes suns and their planetary systems into orbiting star clusters, into galaxies, etc.
\end{remark}

\section{Configuration Pairing}

Throughout, assume that all graphs and Lie bracket expressions have labels and letters from 
the same alphabet.

\begin{definition}
Given a graph $G$ and a Lie bracket expression $L$, a bijection $\sigma:G\leftrightarrow L$ is 
a bijection from the vertices of $G$ to the positions in the Lie bracket expression $L$ compatible
with labels and letters 
(vertices of $G$ are sent to positions in $L$ labeled with the identical letter).
\end{definition}

\begin{example} Following are some basic examples investigating bijections between 
graphs and Lie bracket expressions. 

$\bullet \quad$ There are no bijections 
  $\graphpp{a}{b}{a} \leftrightarrow [a,b]$ because there are three vertices
  in the graph but only two positions in the Lie bracket expression.
  \\ \phantom{x} \qquad
  Similarly, there are no bijections 
  $\linep{a}{b} \leftrightarrow [[a,b],a]$ 
  
$\bullet \quad$ There are no bijections 
  $\graphpp{a}{b}{c} \leftrightarrow [[a,b],a]$ because there is no letter
  $c$ in the Lie bracket expression. 
  
$\bullet \quad$ There is only one bijection 
  $\graphpp{a}{b}{c} \leftrightarrow [[b,c],a]$  given by identifying each vertex
  with the correspondingly labeled position in the bracket expression. 
  
$\bullet \quad$ There are two bijections 
  $\graphpp{a}{b}{a} \leftrightarrow [[a,b],a]$  since there are two ways 
  to choose an identification between the two vertices $a$ and the two bracket positions $a$.
  
$\bullet \quad$ More generally, there are $n!$ bijections 
  $\bigstar (a^n b) \leftrightarrow \lgbrack a^n b\rgbrack$.
\end{example}

Given a graph $G$ and a subset $V$ of the vertices of $G$, write $|V|$ for the full
subgraph of $G$ with vertices from $V$; i.e. two vertices are connected by an edge in 
$|V|$ if and only if they are connected by an edge in $G$.  Recall that a graph is connected if 
every two vertices can be connected by a path of edges.  We will say that directed graphs are
connected if they are connected, ignoring edge directions.

The configuration pairing defined in \cite{Sin06} between directed graphs and rooted trees gives a 
pairing between graphs and Lie bracket expressions which can be defined as follows \cite{Wal10}.

\begin{definition}
Given a graph $G$ and a Lie bracket expression $L$ as well as a bijection 
$\sigma: G \leftrightarrow L$, the $\sigma$-configuration
pairing of $G$ and $L$ is 
\[
\langle G,\ L\rangle_\sigma = 
\begin{cases}
  0,  & 
    \parbox[t]{.6\textwidth}{if $L$ contains a sub-bracket expression $[H,\,K]$ so that
  			the corresponding subgraphs $|\sigma^{-1}H|$ and $|\sigma^{-1}K|$ are not 
            connected graphs with exactly one edge between them in $G$}\\
  (-1)^n, & 
    \parbox[t]{.6\textwidth}{otherwise (where $n$ is the number of edges 
            of $G$ whose orientation corresponds under $\sigma$ to the
            right-to-left orientation of positions in $L$).}
 \end{cases}
\]
The configuration pairing of $G$ and $L$ is the sum over all bijections $\sigma$.
\[
\langle G,\ L \rangle = 
   \sum_{\sigma:G \leftrightarrow L} \langle G,\ L\rangle_\sigma
\]
\end{definition}

Casually, we will say that an edge 
			 $\begin{aligned}\begin{xy}
				(0,-2)*[o]+=<7pt>{\scriptstyle a}="a",
				(6,2)*[o]+=<7pt>{\scriptstyle b}="b", 
			    "a";"b"**\dir{-}?>*\dir{>},
			 \end{xy}\end{aligned}$ 
in $G$ 
whose orientation corresponds under $\sigma$ to the right-to-left orientation
of L (i.e. $\sigma(a)$ is to the right of $\sigma(b)$ in $L$)
``moves leftwards in $L$ under $\sigma$''.
 
\begin{example}
Following are some example computations of configuration pairings.

$\bullet \quad \Bigl\langle \graphpp{a}{b}{c},\ [[b,c],a] \Bigr\rangle = -1$.  
  \\ 
  There is only one bijection.  
  In this bijection only the edge $\linep{a}{b}$  
  moves leftwards in $[[b,c],a]$.
  
\vspace{.2em} 

$\bullet \quad \Bigl\langle \graphpp{a}{a}{b},\ [[a,b],a] \Bigr\rangle = -1 -1 = -2$.
  \\ 
  There are two bijections, each making one edge (either the edge $\linep{a}{a}$ or 
  the edge $\dlinep{a}{b}$) move leftwards in $[[a,b],a]$.
  
\vspace{.2em} 

$\bullet \quad \Bigl\langle \graphpp{a}{b}{a},\ [[a,a],b] \Bigr\rangle = 0$.
  \\ 
  For each of the two bijections, $|\sigma^{-1}\bigl([a,a]\bigr)|$ 
  is disconnected in $G$.
  
\vspace{.2em} 
  
$\bullet \quad \Bigl\langle \longgraph{a}{b}{c}{a},\ [[a,b],[a,c]] \Bigr\rangle = -1 + 1 = 0$. 
  \\ 
  There are two bijections.
  One bijection makes $\linep{c}{a}$ go leftwards.  The other bijection makes $\linep{a}{b}$ 
  and also $\linep{c}{a}$ go leftwards.
  
\vspace{.2em} 
  
$\bullet \quad$ The pairing of a linear (or ``long'') graph 
  $\longgraph{a_1}{a_2}{\vphantom{x}\dots}{a_n}$ with a bracket expression $L$ is equal to the 
  coefficient of the $(a_1a_2\dots a_n)$ term in the associative polynomial for $L$. \cite{Wal10}
\end{example}

The configration pairing encodes a duality between free Lie algebras and
graphs modulo the Arnold and arrow reversing identities \cite{SiWa11}.  
In the current work we will 
use only that the configuration pairing is well defined on Lie algebras -- 
i.e. the configuration pairing
vanishes on Jacobi and anti-commutativity Lie bracket expressions.  Thus the configuration
pairing with graphs can be used to distinguish Lie bracket expressions, and in particular
can be used to establish linear independence. 

The main theorem will be proven essentially via recursive application of the following proposition
whose proof is trivial.

\begin{proposition}\label{P:base case}
Let $w_1$ and $w_2$ be Lyndon-Shirshov words.
If $w_1 = a^n b$ is $a$-simple then 
\[\bigl\langle \bigstar (w_1),\ \lgbrack w_2 \rgbrack\bigr\rangle = 
 \begin{cases} 
     n! & \mbox{if }w_2 = w_1, \\
     0 & \mbox{otherwise.}
 \end{cases} 
\]
A similar result holds if $w_2$ is $a$-simple. 
\end{proposition}
\begin{proof}
Suppose that $w_1$ and $w_2$ are Lyndon-Shirshov words with 
$\bigl\langle \bigstar(w_1),\ \lgbrack w_2 \rgbrack \bigr\rangle \neq 0$.
Note that $w_1$ and $w_2$ must be written with the same letters 
for any bijections 
$\sigma: \bigstar(w_1) \leftrightarrow \lgbrack w_2 \rgbrack$ to exist.  Furthermore
$w_1$ and $w_2$ must share the same initial letter, since Lyndon-Shirshov words always begin
with their lowest-ordered letter.  Thus $w_1 = w_2.$  

If $w_1 = w_2$, then there are $n!$ possible bijections 
$\sigma:\bigstar(a^n b) \leftrightarrow \lgbrack a^n b\rgbrack$.
For each of these $\bigl\langle \bigstar(a^n b),\ \lgbrack a^n b\rgbrack \bigr\rangle_\sigma = 1$.
\end{proof}

\section{The Basis Theorem}

\begin{theorem}\label{T:main}
If $w_1$ and $w_2$ are Lyndon-Shirshov words, then
$\bigl\langle \bigstar(w_1),\ \lgbrack w_2 \rgbrack\bigr\rangle \neq 0$ if and only if $w_1 = w_2$ 
(in this case it is a product of factorials).
\end{theorem}

Our desired result follows as a simple corollary.

\begin{corollary}
The left-greedy bracketing of Lyndon-Shirshov words gives a basis for free Lie algebras.
\end{corollary}
\begin{proof}
A perfect pairing of graphs with left-greedy brackets of Lyndon-Shirshov words implies that 
the left-greedy brackets of Lyndon-Shirshov words are
linearly independent. Since the number of Lyndon-Shirshov words of length $n$ equals the dimension of
the vector space of length $n$ Lie bracket expressions, this is enough to show that left-greedy
brackets of Lyndon-Shirshov words form a basis for the free Lie algebra.
\end{proof}

Now we will prove the main theorem.

\begin{proof}[Proof of Theorem~\ref{T:main}]
Suppose that $w_1$ and $w_2$ are 
Lyndon-Shirshov words with nonzero pairing
\[\bigl\langle \bigstar(w_1),\ \lgbrack w_2 \rgbrack \bigr\rangle \neq 0.\]
Fix a bijection $\sigma:\bigstar(w_1) \leftrightarrow \lgbrack w_2 \rgbrack$.
We will show that $w_1 = w_2$ by inducting on depth of the nested partition resulting from 
fully partitioning the Lyndon-Shirshov words $w_1$ and $w_2$. 

First note, as in the proof of Proposition~\ref{P:base case}, that 
$w_1$ and $w_2$ must be written with the same letters and must share the same initial 
letter, call it ``$a$".  
Thus $w_1$ and $w_2$ both fully partition
where the innermost partitions
are $a$-simple words.

Write $w_2 \twoheadrightarrow (a^{n_1} b_1)(a^{n_2} b_2) \dots (a^{n_k} b_k)$ 
for the innermost partition of $w_2$. 
According to its recursive definition, the bracket expression
$\lgbrack w_2 \rgbrack$ will have sub-bracket expressions $\lgbrack a^{n_i} b_i\rgbrack$.
From the definition of the configuration pairing, these must correspond under $\sigma$ 
to connected, disjoint
subgraphs of $\bigstar(w_1)$.  
However, the only possible connected subgraph of a star graph (with initial letter $a$) 
using the letters
$a^{n_i} b_i$ is $\bigstar(a^{n_i}b_i)$.  
Note that this implies $w_1$ is composed of the subwords $(a^{n_i}b_i)$
(though possibly written in a different order).  Furthermore, the first subword of $w_1$
must be $(a^{n_1}b_1)$ (just as in $w_2$), 
because Lyndon-Shirshov words must begin with their lowest ordered subword.

The induction step is equivalent to the previous case, treating subwords as letters.
At the end of the previous case, for each simple subword $u$ of $w_2$ the sub-bracket expressions 
$\lgbrack u\rgbrack$ of $\lgbrack w_2\rgbrack$ correspond to disjoint connected subgraphs
$\bigstar (u)$ of $\bigstar(w_1)$.  Furthermore, the initial subword of $w_2$ coincides
with the initial subword of $w_1$.

\medskip

To finish the proof, we must note that 
$\bigl\langle \bigstar(w),\ \lgbrack w\rgbrack \bigr\rangle \neq 0$ when $w$ is a Lyndon-Shirshov word.
This is clear since all bijections $\sigma:\bigstar(w) \leftrightarrow \lgbrack w \rgbrack$
have $\bigl\langle \bigstar(w), \lgbrack w\rgbrack\bigr\rangle_\sigma > 0$.
In fact, a few short computations show that 
\[
\begin{aligned}
\bigl\langle \bigstar(a^n b),\ \lgbrack a^n b\rgbrack \bigr\rangle &= n! \\
 \Bigl\langle 
   \bigstar\bigl((a^{n_1}b_1)^m (a^{n_2}b_2)\bigr),\ \lgbrack (a^{n_1}b_1)^m (a^{n_2}b_2) \rgbrack
 \Bigr\rangle
   &= m!\, (n_1!)^m\, n_2! \\
 \mathrm{etc.}
\end{aligned}
\]
\end{proof}

\section{Projection onto the Left-Greedy Basis}

The previous theorem~\ref{T:main} is of independent interest, because it gives a 
direct, computational method 
for writing Lie bracket elements in terms of the left-greedy Lyndon-Shirshov basis via projection.

Given a Lie bracket expression $L$ write $\{w_k\}$ for the set of Lyndon-Shirshov words written
using the 
letters in $L$ (with multiplicity).  Left-greedy brackets of Lyndon-Shirshov words form a 
linear basis, so it is
possible to write $L$ as a linear combination of the $\lgbrack w_k \rgbrack$:
\[ L = c_1\, \lgbrack w_1\rgbrack + \cdots + c_n \lgbrack w_n \rgbrack. \]
We may compute the constants $c_k$ by pairing with 
$\bigstar(w_k)$ since 
$\bigl\langle \bigstar(w_k),\, \lgbrack w_j \rgbrack \bigr\rangle = 0$ for $j \neq k$
by Theorem~\ref{T:main}.  This proves the following.

\begin{corollary}
  Given a Lie bracket expression $L$, the following formula holds
  \[ L = \!\! \sum_{\substack{w\text{\ a}\\ \text{Lyndon-Shirshov word}}} \!\! 
                        \frac{\bigl\langle \bigstar(w),\ L\bigr\rangle}
                             {\bigl\langle \bigstar(w),\ \lgbrack w\rgbrack\bigr\rangle} \  \lgbrack w\rgbrack.\]
\end{corollary}

Recall that the denominators 
${\bigl\langle \bigstar(w),\ \lgbrack w\rgbrack\bigr\rangle}$ are products
of factorials.  Interestingly, each coefficient in the expression above must be an
integer (despite their large denominator). 

Pairing computations are aided by the bracket/cobracket compatibility property of the 
configuration pairing.  Bracket/cobracket compatibility states that pairings 
of a graph $G$ with a 
bracket expression $[L,K]$ may be computed by calculating pairings of $L$ and $K$ with 
the subgraphs obtained by cutting $G$ into two pieces by removing an edge.
The following is Prop. 3.14 of \cite{SiWa11}. 

\begin{proposition}
Bracketing Lie expressions is dual to cutting graphs 
 \[\bigl\langle G,\ [H, K] \bigr\rangle =  \sum_e 
 \bigl\langle G^{\hat e}_1,\ H\bigr\rangle \cdot \bigl\langle G^{\hat e}_2,\ K\bigr\rangle
 \ - \ 
 \bigl\langle G^{\hat e}_1,\ K\bigr\rangle \cdot \bigl\langle G^{\hat e}_2,\ H\bigr\rangle\]
 where $G^{\hat e}_1$ and $G^{\hat e}_2$ are the graphs obtained by removing edge $e$ from 
 $G$, ordered so that $e$ pointed from $G^{\hat e}_1$ to $G^{\hat e}_2$ in $G$.
\end{proposition}

\begin{remark}
Applying bracket/cobracket duality and the definition of the configuration pairing yields a 
recursive method for computation of $\langle G,\ L\rangle$.  Consider the outer-most bracketing 
$L = [H,\ K]$. Look for edges of $G$ which can be removed to separate $G$ into subgraphs 
$G^{\hat e}_1$ and $G^{\hat e}_2$ whose sizes matches that of $H$ and $K$, and check that 
the subgraphs are written using 
the same letters as $H$ and $K$.  
If this is not possible, then the bracketing is 0.
Otherwise the bracketing is given by summing 
$\bigl\langle G^{\hat e}_1,\ H\bigr\rangle \cdot \bigl\langle G^{\hat e}_2,\ K\bigr\rangle$
(or the negative
$\ - \ 
 \bigl\langle G^{\hat e}_1,\ K\bigr\rangle \cdot \bigl\langle G^{\hat e}_2,\ H\bigr\rangle$
if $e$ pointed so that $G^{\hat e}_1$ corresponded to $K$ instead of $H$)
over all such edges.
Recurse.  Note that removing an edge from a star graph will always result in subgraphs which 
are themselves star graphs (though possibly not star graphs of Lyndon-Shirshov words).
\end{remark}

\begin{example}
Consider the Lie bracket expression $L = [[[a,b],b],[[a,b],a]]$.  There are 
three Lyndon-Shirshov words with the letters $aaabbb$.  These words, along with their partition,
left-greedy bracketings, and values of $\bigl\langle \bigstar(w),\  \lgbrack w\rgbrack \bigr\rangle$ are
as follows.

$\bullet \quad aaabbb$ which partitions as 
    $\underline{ \underline{ \underline{aaab}\, \underline{b}}\, \underline{\underline{b}}}$ 
    \\ \phantom{x} \qquad
    with left-greedy bracketing
    $\lgbrack aaabbb\rgbrack = \bigl[\bigl[ [a,[a,[a,b]]],\;b\bigr],\; b\bigr] $ 
    \\ \phantom{x} \qquad 
    and $\bigl\langle \bigstar(aaabbb),\ \lgbrack aaabbb\rgbrack \bigr\rangle = 3!$.

$\bullet \quad aababb$ which partitions as 
    $\underline{\underline{ \underline{aab}\;\underline{ab}}\;\underline{\underline{b}}}$
    \\ \phantom{x} \qquad
    with left-greedy bracketing
    $\lgbrack aababb \rgbrack = \bigl[\bigl[ [a,[a,b]],\; [a,b]\bigr],\; b\bigr] $ 
    \\ \phantom{x} \qquad 
    and $\bigl\langle \bigstar(aababb),\ \lgbrack aababb\rgbrack \bigr\rangle = 2!$.

$\bullet \quad aabbab$ which partitions as
	$\underline{\underline{ \underline{aab}\; \underline{b}}\; \underline{\underline{ab}}}$
    \\ \phantom{x} \qquad
    with left-greedy bracketing
    $\lgbrack aabbab \rgbrack = \bigl[\bigl[ [a,[a,b]],\; b\bigr],\; [a,b]\bigr]$ 
    \\ \phantom{x} \qquad 
    and $\bigl\langle \bigstar(aabbab),\ \lgbrack aabbab\rgbrack \bigr\rangle = 2!$.
    
The configuration pairings with $L$ are as follows.

%
$\bullet \quad \bigl\langle \bigstar(aaabbb),\ [[[a,b],b],[[a,b],a]] \bigr\rangle = 0$, 
  because no single edge of $\bigstar(aaabbb)$ can be removed to separate it into subgraphs
  one of which has a single $a$ and two $b$'s (corresponding to the sub-bracket $[[a,b],b]$).

%
$\bullet \quad \bigl\langle \bigstar(aababb),\ [[[a,b],b],[[a,b],a]] \bigr\rangle = 2$, 
 because only the edge connecting $\bigstar(aab)$ to the remainder of the graph cuts 
 $\bigstar(aababb)$ appropriately.  This reduces the computation to 
 \[- \bigl\langle \bigstar(aab),\ [[a,b],a]\bigr\rangle \cdot 
     \bigl\langle \bigstar(abb),\ [[a,b],b]\bigr\rangle = -(-2)\cdot 1 = 2.\]

%
$\bullet \quad \bigl\langle \bigstar(aabbab),\ [[[a,b],b],[[a,b],a]] \bigr\rangle = -2$, 
 because only the edge connecting $\bigstar(aab)$ to the remainder of the graph cuts 
 $\bigstar(aabbab)$ appropriately.  The computation reduces similarly.
    
Thus $L = \lgbrack aababb \rgbrack - \lgbrack aabbab \rgbrack.$
\end{example}

\bibliographystyle{plain}
\bibliography{bibliography.bib}

\end{document}